\documentclass{article}

\usepackage{amssymb}
\usepackage{amsthm}
\usepackage{amsmath}

\usepackage{a4wide}
\usepackage{float}

\usepackage{graphics}
\usepackage{epsfig}
\usepackage{subfig}

\usepackage{natbib}
\usepackage{setspace}

\newtheorem{theorem}{Theorem}

\theoremstyle{definition}
\newtheorem{definition}[theorem]{Definition}

\providecommand{\url}[1]{\texttt{#1}}
\expandafter\ifx\csname urlstyle\endcsname\relax
\else
\fi

\begin{document}

\title{Dengue in Cape Verde: vector control and vaccination\thanks{This is a preprint of a paper 
whose final and definite form will appear in \emph{Mathematical Population Studies}. 
Paper submitted 03-Oct-2011; revised several times; accepted for publication 2-April-2012.}}

\author{Helena Sofia Rodrigues$^1$
\and M. Teresa T. Monteiro$^2$
\and Delfim F. M. Torres$^3$}

\footnotetext{Address correspondence to Helena Sofia Rodrigues,
School of Business Studies, Viana do Castelo Polytechnic Institute,
4900-347 Viana do Castelo, Portugal. E-mail: sofiarodrigues@esce.ipvc.pt}


\date{$^1$School of Business Studies, Viana do Castelo Polytechnic Institute\\
4900-347 Viana do Castelo, Portugal\\
E-mail: sofiarodrigues@esce.ipvc.pt\\[0.3cm]
$^2$Center Algoritmi\\
Department of Production and Systems, University of Minho\\
4710-057 Braga, Portugal\\
E-mail: tm@dps.uminho.pt\\[0.3cm]
$^3$Center for Research and Development in Mathematics and Applications\\
Department of Mathematics, University of Aveiro\\
3810-193 Aveiro, Portugal\\
E-mail: delfim@ua.pt}

\maketitle


\begin{abstract}
In 2009, for the first time in Cape Verde, an outbreak of dengue was
reported and over twenty thousand people were infected. Only a few
prophylactic measures were taken. The effects of vector control
on disease spreading, such as insecticide (larvicide and adulticide)
and mechanical control, as well as an hypothetical vaccine,
are estimated through simulations with the Cape Verde data.

\medskip


\smallskip

\end{abstract}


\section{Introduction}

Dengue is a disease which is now endemic in Africa, America, Asia,
and the Western Pacific. In Europe there has been no registered cases so far,
yet the main vector of the disease is already present there
and has been followed in Madeira.
According to the World Health Organization \citep{WHO}, the
incidence of dengue has grown drastically in the last decades and
roughly two fifths of the world population is now at risk. So far,
the strategies to control \emph{Aedes aegypti} mosquito proved to be inefficient
and the continuous level of surveillance of the mosquito is low.
Climatic changes, unsanitary habitat, poverty, uncontrolled urbanization,
and global travel favor the propagation of dengue fever. Inadequate water supply
requires large-scale water storages, which are ideal breeding
habitats for mosquitoes. The increasing movement of people and goods
has enabled the dengue virus and its vectors to spread to new parts
of the world \citep{Semenza2009}.
There is no effective vector control and dengue
infection rates have been increasing during the last 40 years,
\textrm{e.g.}, in Thailand has increased from 9/100,000 in 1958
to 189/100,000 in 1998 \citep{Healstead1992,Thailand}.
The virus emerged for the first time
in Cape Verde at the end of September 2009. The outbreak
was the biggest ever recorded in Africa. As the population had never had
contact with the virus, herd immunity was very low. Dengue
type 3 spread throughout four of the nine islands. The worst outbreak occurred
on the island of Santiago, where most people live, up to 1000 cases per day
in November. The Ministry of Health of Cape Verde
reported over 20,000 cases of dengue fever,
which is about 5\% of the total population of the country,
between October and December 2009. From 173 cases of dengue
haemorrhagic fever reported, six people died \citep{CapeVerde,CapeVerdeSaude}.


\section{Dengue and \emph{Aedes aegypti}}
\label{sec:2}

Dengue is transmitted to humans by mosquitoes, mainly
\emph{Aedes aegypti}, and exists either as
\emph{dengue fever} (DF)
or as \emph{dengue hemorrhagic fever} (DHF).
The disease can be contracted by one of
four types of viruses.
Infection with one serotype confers lifetime immunity against that
serotype, but not to the others, and there is evidence that a prior
infection increases the risk of developing DHF
for people infected with another serotype. DF is
characterized by sudden high fever (3 to 7 days) without
respiratory symptoms, intense headache, and painful
joints and muscles. The haemorrhagic form is also characterized by
sudden fever, nausea, vomiting, and fainting due to low blood
pressure by fluid leakage. It usually lasts between two and three
days and can lead to death. There is no specific
effective treatment for dengue. Fluid replacement therapy is used
if clinical diagnosis is made early. Vaccine candidates are undergoing
clinical trials \citep{WHO}.

\emph{Aedes aegypti} is closely associated with humans and their dwellings,
thriving in crowded cities and biting primarily during day light.
Humans also provide nutrients needed for mosquitoes to reproduce
through water-holding containers, in and around human homes.
In urban areas, \emph{Aedes} mosquitoes breed on water
collections in artificial containers such as plastic cups, used
tires, broken bottles, and flower pots. With urbanization, crowded cities,
poor sanitation, and lack of hygiene, environmental conditions foster
the spread of the disease which, even in the absence of fatal forms, breeds
significant economic and social costs (absenteeism, immobilization,
debilitation, and medication) \citep{Derouich2006,Ferreira2006}.

Female mosquitoes acquire infection by taking a blood meal from an
infected human. These infected mosquitoes pass the disease to
susceptible humans. Female mosquitoes lay their eggs on inner
wet walls of containers. Larvae hatch when water inundates eggs.
In the following days, the larvae feed on microorganisms and
particulate organic matter. When the larva has acquired enough
energy and size, metamorphosis changes the larva into a pupa.
Pupae do not feed: they just change in form until the adult body,
the fling mosquito, is formed. The newly formed adult emerges
from the water after breaking the pupal skin. The entire life cycle,
from the aquatic phase (eggs, larvae, pupae) to the adult phase,
lasts from 8 to 10 days at room temperature, depending on the level
of feeding \citep{Christophers1960}.

It is very difficult to control or eliminate \emph{Aedes aegypti}
mosquitoes because they quickly adapt to changes in the environment
and they quickly pullulate again after droughts or prophylactic measures.
The transmission thresholds and the extent of dengue
transmission are determined by the level
of herd immunity in the human population to circulating virus
serotype(s), virulence characteristics of the viral strain,
survival, feeding behavior, abundance of \emph{Aedes aegypti},
climate and human density, distribution, and movement \citep{Scott2004}.
There are two primary preventions: larval control and adult mosquito
control, depending on the intended target \citep{Natal2002}.
Larvicide treatment is an effective control of the vector
larvae, together with mechanical control, which is related to
educational campaigns to remove still water from domestic recipients
and eliminating possible breeding sites. The larvicide
should be long-lasting and preferably have World
Health Organization clearance for use in drinking water
\citep{Derouich2003}. The application of adulticides can
drop the mosquito vector population. However, the efficacy
is often constrained by the difficulty in achieving
sufficiently high coverage of resting surfaces \citep{Devine2009}.

While \citet{Feng} investigate the competitive exclusion principle
in a two-strain dengue model, \citet{Choewll} estimate
the basic reproduction number for dengue using spatial
epidemic data. \citet{Tewa}, established the global asymptotic stability of the equilibria
of a single-strain dengue model. \citet{Thome2009} introduce a sterile insect technique.


\section{Model}
\label{sec:3}

We adapt the model presented in \citet{Dumont2008} and \citet{Dumont2010}
to dengue, with the considerations of \citet{Sofia2009,MyID:168,Sofia2010}.
We consider three controls simultaneously: larvicide, adulticide,
and mechanical control, with mutually-exclusive compartments,
to study the outbreak of 2009 in Cape Verde and
improve upon \citet{Sofia2009}. We denote
$S_h$ the total number of susceptible,
$I_h$ the total number of infected and infectious, and
$R_h$ the total number of resistant individuals.
The total human population is a constant $N_h=S_h(t)+I_h(t)+R_h(t)$.
The population is homogeneous, which means that every individual
of a compartment is homogeneously mixed with the other individuals.
Immigration and emigration are ignored.

Female mosquitoes are in total number $A_m$ in the aquatic phase
(including egg, larva, and pupa stages),
$S_m$ are susceptible, and $I_m$ are infected.
Mosquitoes are considered to live too briefly to develop resistance.
Each mosquito has an equal probability to bite any host.
Humans and mosquitoes are assumed to be born susceptible.

We consider three controls: the proportion
$c_{A}$ of larvicide, the proportion $c_{m}$ of adulticide,
and the proportion $\alpha$ of mechanical control.
Larval control targets the immature mosquitoes living in water
before they bite. The natural soil bacterium
\emph{Bacillus thuringiensis israelensis} (Bti) is sprayed
from the ground or by air to larval habitats.
The control of adult mosquitoes is necessary when mosquito populations
cannot be treated in their larval stage. Depending upon the size of the
area, either trucks for ground adulticide treatments or aircraft
for aerial adulticide treatments are used.
The purpose of mechanical control is to reduce
larval habitat areas. The parameters used in the model are:

$N_h$ total population;

$B$ average daily biting (per day);

$\beta_{mh}$ transmission probability from $I_m$ (per bite);

$\beta_{hm}$ transmission probability from $I_h$ (per bite);

$1/\mu_{h}$ average lifespan of humans (in days);

$1/\eta_{h}$ mean viremic period (in days);

$1/\mu_{m}$ average lifespan of adult mosquitoes (in days);

$\varphi$ number of eggs at each deposit per capita (per day);

$1/\mu_{A}$ natural mortality of larvae (per day);

$\eta_{A}$ maturation rate from larvae to adult (per day);

$m$ female mosquitoes per human;

$k$ total number of larvae per human.

The dengue epidemic is modelled by the nonlinear
time-varying state equations
\begin{equation}
\label{ode1}
\begin{cases}
S_h'(t) = \mu_h N_h - \left(B\beta_{mh}\frac{I_m(t)}{N_h}+\mu_h\right)S_h(t)\\
I_h'(t) = B\beta_{mh}\frac{I_m(t)}{N_h}S_h(t) -(\eta_h+\mu_h) I_h(t)\\
R_h'(t) = \eta_h I_h(t) - \mu_h R_h(t)
\end{cases}
\end{equation}
and
\begin{equation}
\label{ode2}
\begin{cases}
A_m'(t) = \varphi \left(1-\frac{A_m(t)}{\alpha k N_h}\right)(S_m(t) + I_m(t))
-\left(\eta_A + \mu_A + c_A\right) A_m(t)\\
S_m'(t) = \eta_A A_m(t)
-\left(B \beta_{hm}\frac{I_h(t)}{N_h}+\mu_m + c_m\right) S_m(t)\\
I_m'(t) = B \beta_{hm}\frac{I_h(t)}{N_h}S_m(t)
-\left(\mu_m + c_m\right) I_m(t)
\end{cases}
\end{equation}
with the initial conditions
\begin{equation*}
\begin{tabular}{llll}
$S_h(0)=S_{h0},$ &  $I_h(0)=I_{h0},$ &
$R_h(0)=R_{h0},$ \\
$A_m(0)=A_{m0},$ & $S_{m}(0)=S_{m0},$ & $I_m(0)=I_{m0}.$
\end{tabular}
\end{equation*}

Figure~\ref{model} shows a scheme of the model.
\begin{figure}
\begin{center}
\includegraphics[scale=0.5]{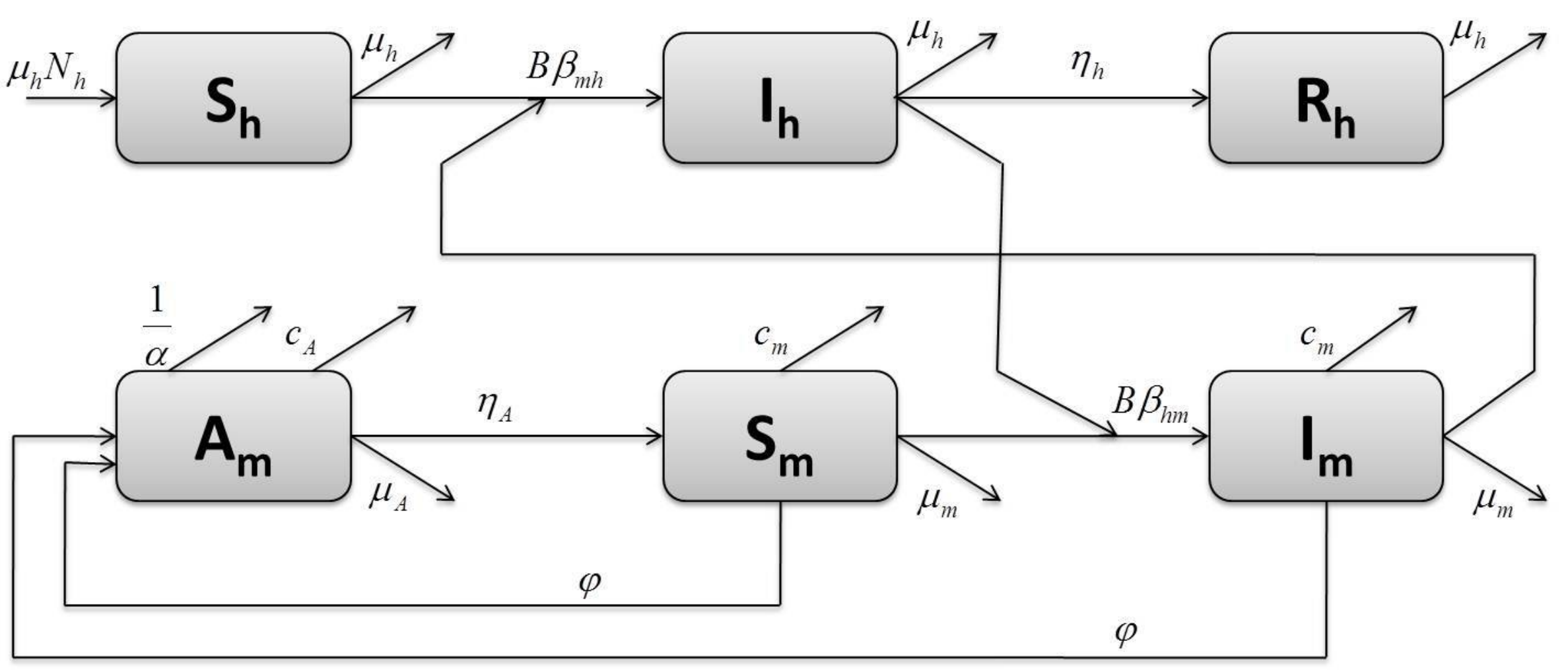}
\end{center}
\caption{Epidemiological model SIR (Susceptible, Infected, Recovered)
$+$ ASI (Aquatic phase, Susceptible, Infected)
with $S_h$, $I_h$, $R_h$, $A_m$, $S_m$, and $I_m$
as state variables, $c_A$, $c_m$, and $\alpha$ as controls,
and $N_h$, $B$, $\beta_{mh}$, $\beta_{hm}$, $\mu_h$, $\eta_h$,
$\mu_m$, $\varphi$, $\mu_A$, $\eta_A$, $m$, and $k$ as given parameters. \label{model}}
\end{figure}


\section{Equilibrium points and the basic reproduction number}

We study the solutions of System \eqref{ode1}--\eqref{ode2} in the closed set
\begin{equation*}
\Omega=\left\{(S_h,I_h,R_h,A_m,S_m,I_m)\in \mathbb{R}_{+}^{6}:
S_h+I_h+R_h\leq N_h,\,  A_m\leq k N_h, \, S_m+I_m \leq mN_h\right\}.
\end{equation*}
The $\Omega$ set is positively invariant
with respect to Eq.~\eqref{ode1}--\eqref{ode2} \citep{Sofia2011}.
System \eqref{ode1}--\eqref{ode2} has at most three biologically
meaningful equilibrium points (\textrm{cf.} Theorem~\ref{thm:thm1}).

\begin{definition}
A sextuple $E = \left(S_h,I_h,R_h,A_m,S_m,I_m\right) \in \mathbb{R}^6$ is an
\emph{equilibrium point} for System \eqref{ode1}--\eqref{ode2}
if it satisfies
\begin{equation}
\label{equilibrio}
\begin{cases}
\mu_h N_h - \left(B\beta_{mh}\frac{I_m}{N_h}+\mu_h\right)S_h=0\\
B\beta_{mh}\frac{I_m}{N_h}S_h -(\eta_h+\mu_h) I_h=0\\
\eta_h I_h - \mu_h R_h=0\\
\varphi \left(1-\frac{A_m}{\alpha k N_h}\right)(S_m+I_m)-(\eta_A+\mu_A+c_A) A_m=0\\
\eta_A A_m - \left(B \beta_{hm}\frac{I_h}{N_h}+\mu_m + c_m\right) S_m=0\\
B \beta_{hm}\frac{I_h}{N_h}S_m -(\mu_m + c_m) I_m=0.
\end{cases}
\end{equation}
An equilibrium point $E$ is \emph{biologically meaningful}
if and only if $E \in \Omega$. The biologically meaningful equilibrium points
are said to be disease-free or endemic depending on $I_h$ and $I_m$:
if there is no disease for both populations of humans and mosquitoes ($I_h=I_m=0$),
then the equilibrium point is a \emph{disease-free equilibrium} (DFE);
otherwise, if $I_h > 0$ or $I_m > 0$, then the equilibrium point is called \emph{endemic}.
\end{definition}

\begin{theorem}
\label{thm:thm1}
System \eqref{ode1}--\eqref{ode2} admits at most three biologically
meaningful equilibrium points, at most two DFE points,
and at most one endemic equilibrium point. Let
\begin{equation}
\label{eq:M}
\mathcal{M} =-\left(\eta_A \mu_m+\eta_A c_m+\mu_A \mu_m+\mu_A
c_m+c_A \mu_m+c_A c_m-\varphi \eta_A\right),
\end{equation}
\begin{equation}
\label{eq:xi:chi}
\xi = \varphi(\mu_m+c_m)^{2}(\eta_h+\mu_h), \quad
\chi = \alpha k B^2 \beta_{hm}\beta_{mh}\mathcal{M},
\end{equation}
\begin{equation}
\label{eq:E1:E2}
E_{1}^{*}=\left(N_h,0,0,0,0,0\right), \quad
E_{2}^{*}=\left(N_h,0,0,\frac{\alpha
k N_h \mathcal{M}}{\eta_A\varphi}, \frac{\alpha k N_h
\mathcal{M}}{(\mu_m+c_m) \varphi},0\right),
\end{equation}
and $E_{3}^{*}=\left(S_h^*,I_h^*,R_h^*,A_m^*,S_m^*,I_m^*\right)$ with
\begin{equation}
\label{eqE3:t2}
\left\{
\begin{split}
S_{h}^{*} &= \frac{ -\varphi
N_h(\mu_m\eta_h+\mu_h B
\beta_{hm}+c_m\mu_h+c_m\eta_h+\mu_m\mu_h)(\mu_m+c_m)}{B\beta_{hm}(-\alpha
k B \beta_{mh}\mathcal{M}-\varphi\mu_h(\mu_m+c_m))},\\
I_{h}^{*} &=\frac{\mu_h
N_h\left(\xi-\chi\right)}{(\eta_h+\mu_h)B\beta_{hm}(-\alpha
k B \beta_{mh}\mathcal{M}-\varphi\mu_h(\mu_m+ c_m))},\\
R_{h}^{*} &= \frac{\eta_h
N_h\left(\xi-\chi\right)}{(\eta_h+\mu_h)B\beta_{hm}\left(-\alpha
k B \beta_{mh}\mathcal{M}-\varphi \mu_h(\mu_m+c_m)\right)},\\
A_{m}^{*} &= \frac{N_h k \alpha \mathcal{M}}{\varphi \eta_A},\\
S_{m}^{*} &=\frac{N_h\mu_h(c_m+\mu_h)(\mu_h+\eta_h)}{B \beta_{mh}(\mu_m \eta_h+\mu_h B
\beta_{hm}+c-m\mu_h+c_m\eta_h+\mu_m\mu_h)}\\
&\qquad -\frac{N_h\alpha k \left(c_m(\mu_h+\eta_h)(\mu_A+\eta_A+c_A)
+\mu_m(\mu_h(\eta_A+\mu_A)+\eta_{h}(c_A+\eta_A))\right)}{\varphi(\mu_m \eta_h+\mu_h B
\beta_{hm}+c-m\mu_h+c_m\eta_h+\mu_m\mu_h)}\\
&\qquad -\frac{N_h\alpha k \left(-\eta
\varphi (\mu_h+\eta_h)+\mu_m(\eta_h\mu_A+\mu_h
c_A)\right)}{\varphi(\mu_m \eta_h+\mu_h B
\beta_{hm}+c-m\mu_h+c_m\eta_h+\mu_m\mu_h)},\\
I_{m}^{*} &=\frac{-\mu_h N_h\left(\xi-\chi\right)}{B\beta_{mh}\left(
\varphi B\beta_{hm}\mu_h(\mu_m+c_m)+\xi\right)}.
\end{split}
\right.
\end{equation}
If $\mathcal{M} \le 0$, then there is only
one biologically meaningful equilibrium point $E_{1}^{*}$,
which is a DFE point. If $\mathcal{M} > 0$ with $\xi \ge \chi$,
then there are two biologically meaningful equilibrium points
$E_{1}^{*}$ and $E_{2}^{*}$, both DFE points.
If $\mathcal{M} > 0$ with $\xi < \chi$,
then there are three biologically meaningful equilibrium points
$E_{1}^{*}$, $E_{2}^{*}$, and $E_{3}^{*}$,
where $E_{1}^{*}$ and $E_{2}^{*}$ are DFEs
and $E_{3}^{*}$ is endemic.
\end{theorem}

\begin{proof}
System \eqref{equilibrio} has four solutions
obtained with the software Maple:
$E_{1}^{*}$, $E_{2}^{*}$, $E_{3}^{*}$, and $E_{4}^{*}$.
The equilibrium point $E_{1}^{*}$ is always a DFE
because it always belong to $\Omega$ with $I_h=I_m=0$.
In contrast, $E_{4}^{*}$ is never biologically realistic
because it always has some negative coordinates.
The other two equilibrium points, $E_{2}^{*}$ and $E_{3}^{*}$,
are biologically realistic only for certain values of the parameters.
The equilibrium $E_{2}^{*}$ is biologically realistic if and only if
$\mathcal{M} \ge 0$, in which case it is a DFE. For the
condition $\mathcal{M} \le 0$, the third
equilibrium $E_{3}^{*}$ is not biologically realistic.
If $\mathcal{M}>0$, then three situations can occur
with respect to $E_{3}^{*}$: if
$\xi = \chi$, then $E_{3}^{*}$ degenerates into $E_2^*$,
which means that $E_{3}^{*}$ is the DFE $E_{2}^{*}$;
if $\xi > \chi$, then $E_{3}^{*}$ is not biologically realistic; otherwise,
one has $E_{3}^{*} \in \Omega$ with $I_h \ne 0$ and $I_m \ne 0$, which means that
$E_{3}^{*}$ is an endemic equilibrium point.
\end{proof}

By algebraic manipulation, $\mathcal{M}>0$ is equivalent to the condition
$\displaystyle \frac{(\eta_A+\mu_A+c_A)(\mu_m+c_m)}{\varphi\eta_A}>1$,
which is related to the number for offspring mosquitos. Thus,
if $\mathcal{M} \leq 0$, then the mosquito population will collapse and the only
equilibrium for the whole system is the trivial DFE $E_{1}^{*}$.
If $\mathcal{M} > 0$, then the mosquito population is sustainable. From a
biological standpoint, the equilibrium $E_{2}^{*}$ is more plausible,
because the mosquito is in its habitat, but without the disease.

\begin{definition}\citep{Hethcote2008}
The \emph{basic reproduction number},
denoted by $\mathcal{R}_0$, is defined as the
average number of secondary infections occurring when one
infective is introduced into a completely susceptible population.
\end{definition}

The basic reproduction number provides an invasion criterion for the
initial spread of the virus in a susceptible population. For this case,

\begin{theorem}
\label{thm:r0}
The basic reproduction number $\mathcal{R}_0$ associated to the
differential System \eqref{ode1}--\eqref{ode2} is
\begin{equation}
\label{eq:R0}
\mathcal{R}_0 = \left(\frac{\alpha k B^2 \beta_{hm} \beta_{mh}
\mathcal{M}}{\varphi (\eta_h + \mu_h) (c_m + \mu_m)^2}\right)^{\frac{1}{2}}
= \left(\frac{\chi}{\xi}\right)^{\frac{1}{2}}.
\end{equation}
\end{theorem}

\begin{proof}
In agreement with \citet{Driessche2002}, we consider the
epidemiological compartments with new infections,
$I_h$ and $I_m$. The two differential equations related
with these two compartments are rewritten as
$x'(t)=\mathcal{F}(x(t))-\mathcal{V}(x(t))$,
where $x=\left(I_h, I_m\right)$, $\mathcal{F}$ is
the rate of production of new infections,
and $\mathcal{V}$ the transition rates between states:
\begin{equation*}
\mathcal{F}(x)
=\left(
\begin{array}{c}
B \beta_{mh} \frac{I_{m}}{N_h}S_{h}  \\
B \beta_{hm} \frac{I_{h}}{N_h}S_{m}\\
\end{array}
\right), \quad
\mathcal{V}(x)=\left(
\begin{array}{c}
(\eta_h+\mu_h)I_{h}  \\
(c_m+\mu_m)I_{m} \\
\end{array}
\right).
\end{equation*}
The Jacobian derivatives are
\begin{equation*}
J_{\mathcal{F}(x)}
=\left(
\begin{array}{cc}
0 &  B \beta_{mh} \frac{S_{h}}{N_h}  \\
B \beta_{hm} \frac{S_{m}}{N_h} & 0\\
\end{array}
\right), \quad
J_{\mathcal{V}(x)}=\left(
\begin{array}{cc}
\eta_h+\mu_h & 0  \\
0 & c_m+\mu_m\\
\end{array}
\right).
\end{equation*}
The quantity $J_{\mathcal{F}(x)} J_{\mathcal{V}(x)}^{-1}$
gives the total number of new infections over the
course of an outbreak. The largest eigenvalue gives the asymptotic
growth of the infected population, giving $\mathcal{R}_0$ as the
spectral radius of the matrix $J_{\mathcal{F}(x)} J_{\mathcal{V}(x)}^{-1}$
in a DFE point. Maple was used to obtain
\begin{equation}
\label{eq:r0:m}
\mathcal{R}_0 = \left(\frac{B^2 \beta_{hm} \beta_{mh} S_{h_{\mathrm{DFE}}}
S_{m_{\mathrm{DFE}}}}{(\eta_h + \mu_h) (c_m + \mu_m) N_h^2}\right)^{\frac{1}{2}}.
\end{equation}
The basic reproduction number $\mathcal{R}_0$ in Eq.~\eqref{eq:R0} is obtained
from replacing $S_{h_{\mathrm{DFE}}}$ and $S_{m_{\mathrm{DFE}}}$ in Eq.~\eqref{eq:r0:m}
by those of the DFE $E_2^*$.
\end{proof}

The model plays on the populations of host and vector,
and the expected basic reproduction number should reflect the
infection transmitted from host to vector and vice-versa.
Accordingly, $\mathcal{R}_0$ can be seen as
$\mathcal{R}_0=(\mathcal{R}_{hm}\times\mathcal{R}_{mh})^{\frac{1}{2}}$. The
infection host-vector is represented by
$\mathcal{R}_{hm}=\frac{B \beta_{hm}S_{m_{\mathrm{DFE}}}}{N_h(\eta_h +
\mu_h)}$, where the term $\frac{B \beta_{hm}S_{m_{\mathrm{DFE}}}}{N_h}$
represents the transmission probability of the disease from humans
to mosquitoes in a susceptible population of
vectors, and the term $\frac{1}{\eta_h + \mu_h}$ the
human viremic period. Analogously, the infection vector-host
is represented by $\mathcal{R}_{mh}
= \frac{B \beta_{mh}S_{h_{\mathrm{DFE}}}}{N_h(c_m + \mu_m)}$,
where $\frac{B \beta_{mh}S_{h_{\mathrm{DFE}}}}{N_h}$
represents the transmission of the disease from mosquito
to the susceptible human population, and $\frac{1}{c_m+\mu_m}$
the lifespan of an adult mosquito.

If $\mathcal{R}_0<1$, on average, an infected individual
produces less than one new infected individual over the course of
its infectious period, and the disease declines. Conversely, if
$\mathcal{R}_0>1$, then each individual infects more than one
person, and the disease invades the population.

\begin{theorem}
\label{thm:thm3}
If $\mathcal{M}>0$ and $\mathcal{R}_{0}>1$,
then System \eqref{ode1}--\eqref{ode2}
admits the endemic equilibrium
$E_{3}^{*}=\left(S_h^*,I_h^*,R_h^*,A_m^*,S_m^*,I_m^*\right)$
given by Eq.~\eqref{eqE3:t2}.
\end{theorem}

\begin{proof}
The only solution of Eq.~\eqref{equilibrio}
with $I_h > 0$ or $I_m > 0$, the only endemic equilibrium, is $E_3^*$.
That occurs, in agreement with Theorem~\ref{thm:thm1},
in the case $\mathcal{M}>0$ and $\chi > \xi$.
The condition $\chi > \xi$ is equivalent,
by Theorem~\ref{thm:r0}, to $\mathcal{R}_{0}>1$.
\end{proof}

Using the methods developed in \citet{Li1996} and \citet{Driessche2002},
if $\mathcal{R}_0 \leq 1$, then the DFE is globally asymptotically stable in
$\Omega$, and the vector-borne disease always dies out;
if $\mathcal{R}_0 > 1$, then the unique endemic equilibrium
is globally asymptotically stable in $\Omega$,
so that the disease, if initially present, persists
at the unique endemic equilibrium level.


\section{Numerical implementation}
\label{sec:4}

In the epidemic of Cape Verde, infections were
rising at a rate of one thousand people a day.
The entire population was asked to participate in the campaign
of cleaning and eradication of the mosquito, with the help
of the police and the army. Data for humans are available at
\citet{INEcv}, but knowledge of mosquitoes in Africa is poor.
For \emph{Aedes aegypti}, \citet{Esteva2005} and \citet{Coelho2008}
have collected observations from Brazil.
The simulations were carried out with
$N_h=480,000$, $B=0.8$, $\beta_{mh}=0.375$,
$\beta_{hm}=0.375$, $\mu_{h}=1/(71\times365)$, $\eta_{h}=1/3$,
$\mu_{m}=1/10$, $\mu_{m}=1/10$, $\varphi=6$, $\mu_A=1/4$,
$\eta_A=0.08$, $m=3$, $k=3$. The initial conditions for the
problem are: $S_{h0}=N_h-10$, $I_{h0}=10$, $R_{h0}=0$,
$A_{m0}=k N_h$, $S_{m0}=m N_h$, $I_{m0}=0$.
With these values, Eq.~\eqref{eq:M} gives $\mathcal{M}>0$.
The time interval is one year and $t_f=365$ days.
We performed all simulations and graphics with Matlab.
To solve System \eqref{ode1}--\eqref{ode2}, we used the {\it ode45} routine.
This function implements a Runge--Kutta method with a variable time step for
efficient computation.

Figures~\ref{humanno} and \ref{mosquitono} show the human and mosquito populations
in the absence of any control. The human infection reached a peak between the 30th and the 50th day.
The infection of the mosquitoes is delayed.
The total number of infected humans obtained from System \eqref{ode1}--\eqref{ode2}
is higher than observations in Cape Verde. The difference can be due to our inability
to quantify individual prophylactic efforts.
\begin{figure}
\begin{center}
\subfloat[\footnotesize{Human population
in the compartments \emph{Susceptible} ($S_h$),
\emph{Infected} ($I_h$), and \emph{Recovered} ($R_h$).}]{\label{humanno}
\includegraphics[scale=0.45]{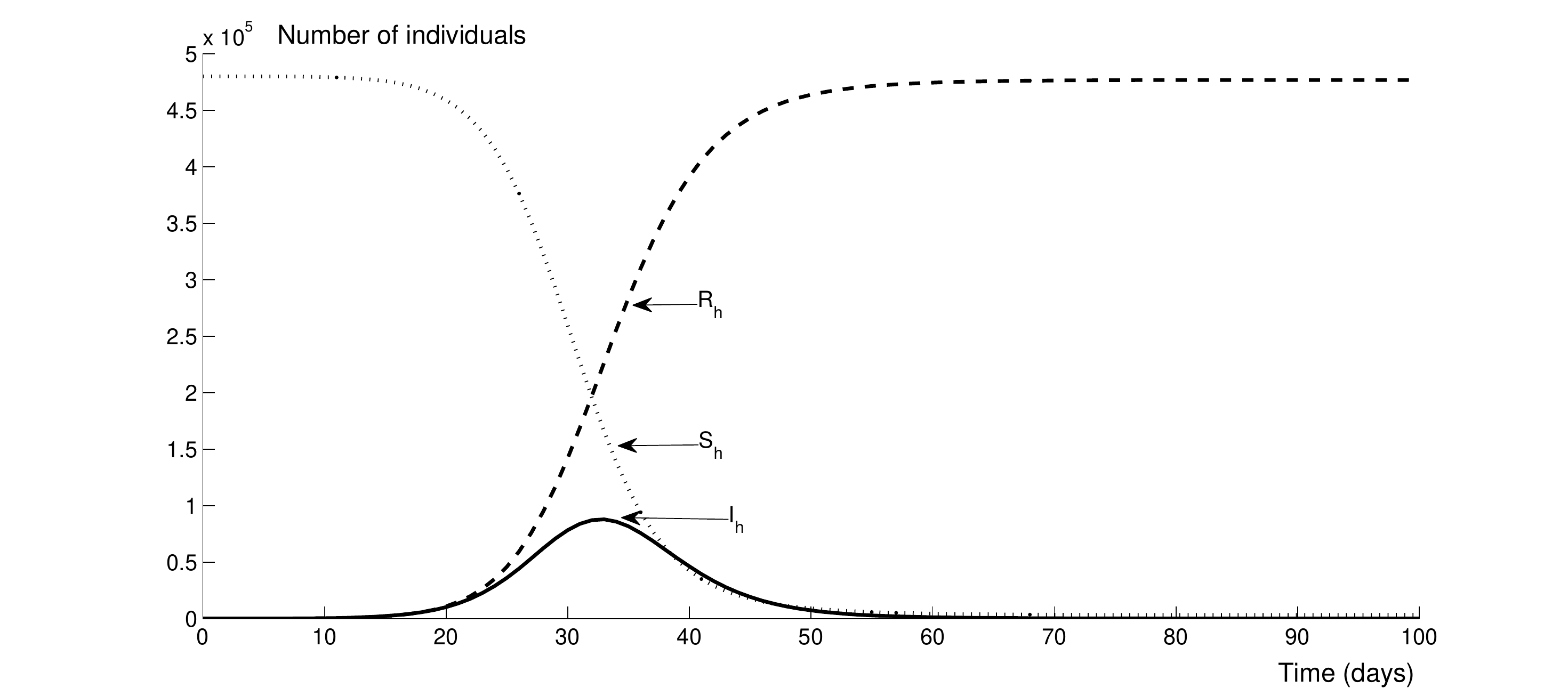}}
\newline
\subfloat[\footnotesize{Mosquito population
in the compartments \emph{Aquatic phase} ($A_m$),
Susceptible ($S_m$), and \emph{Infected} ($I_m$).}]{\label{mosquitono}
\includegraphics[scale=0.45]{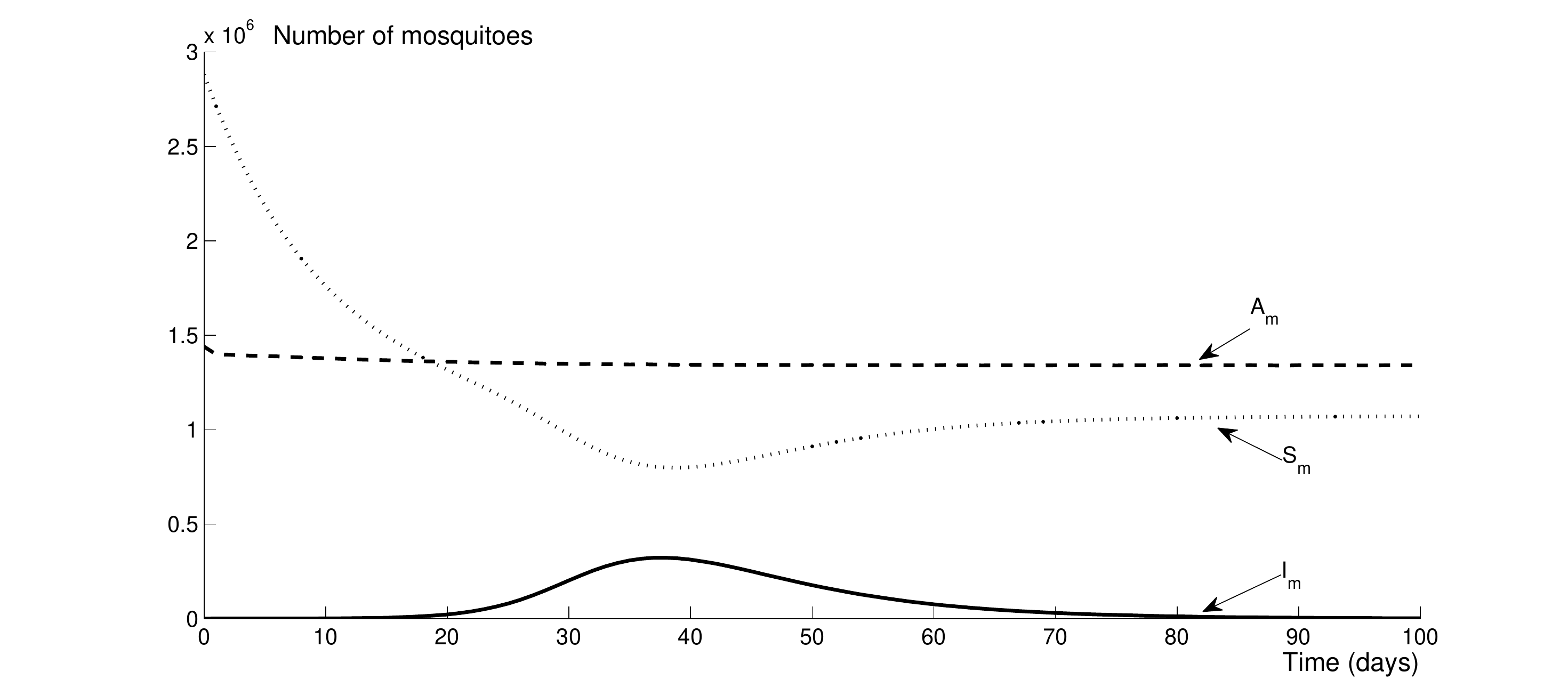}}
\caption{Variation of human and mosquito populations when no control is applied,
\textrm{i.e.}, for System \eqref{ode1}--\eqref{ode2}
with $c_{A} = 0$, $c_{m} = 0$, and $\alpha = 1$.}
\end{center}
\end{figure}

\begin{figure}
\begin{center}
\subfloat[\footnotesize{Variation of infected individuals ($I_h$) with adulticide.}]{
\label{14_human_adulticide_simulation}
\includegraphics[scale=0.45]{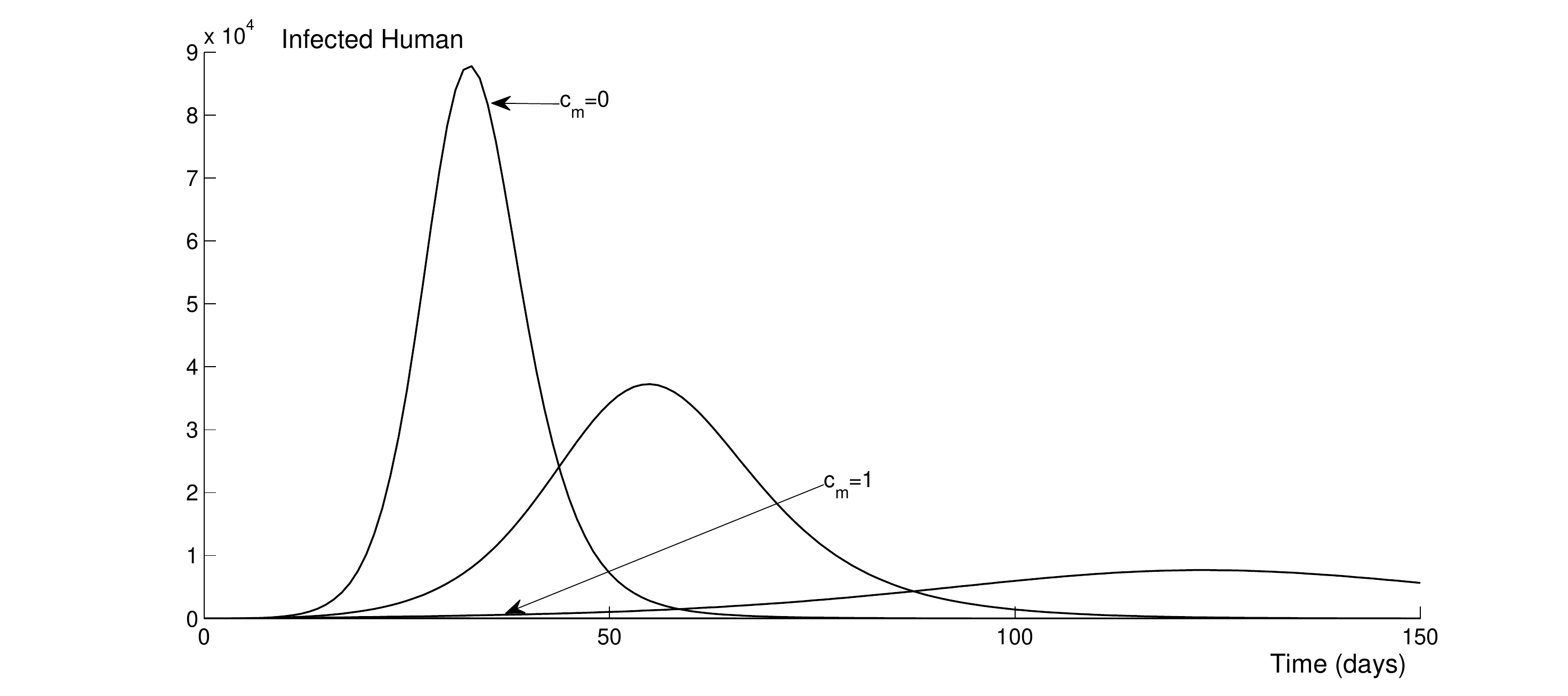}}
\newline
\subfloat[\footnotesize{Variation of infected mosquitoes ($I_m$) with adulticide.}]{
\label{15_mosquito_adulticide_simulation}
\includegraphics[scale=0.45]{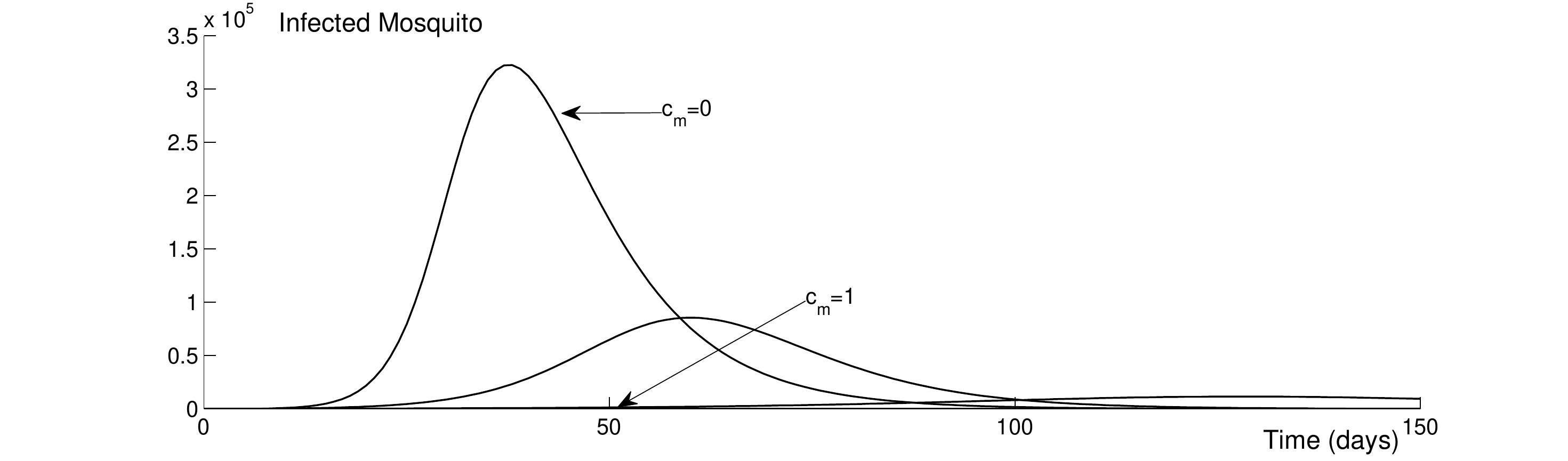}}
\caption{Variation of infected human and infected mosquito populations with
different levels of adulticide: $c_m=0$, $c_m=0.25$,
$c_m=0.50$, $c_m=0.75$, and $c_m=1$.
The other controls are not applied: $c_{A} = 0$, and $\alpha = 1$.}
\end{center}
\end{figure}
\begin{figure}
\begin{center}
\subfloat[\footnotesize{Variation of infected individuals ($I_h$)
with larvicide.}]{\label{12_human_larvicide_simulation}
\includegraphics[scale=0.45]{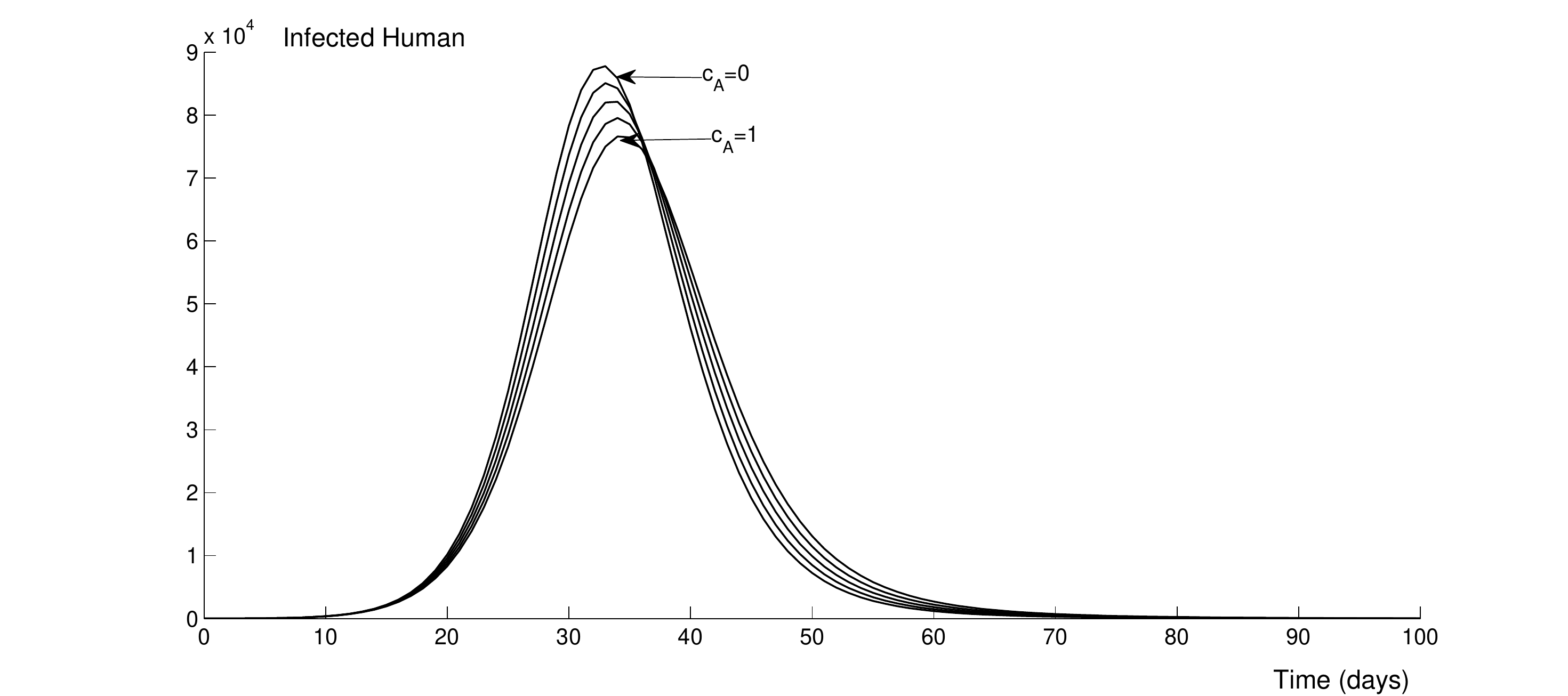}}
\newline
\subfloat[\footnotesize{Variation of infected mosquitoes ($I_m$) with
larvicide.}]{\label{13_mosquito_larvicide_simulation}
\includegraphics[scale=0.45]{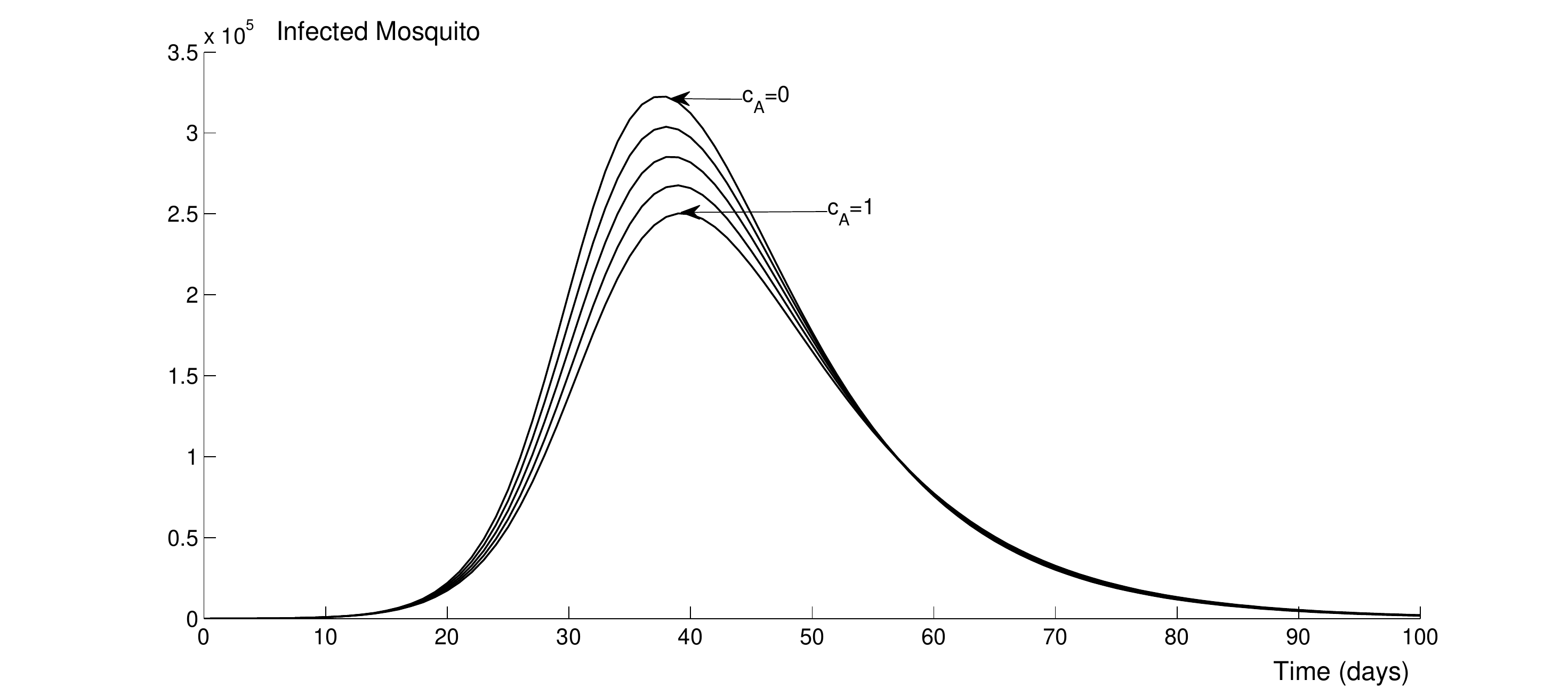}}
\caption{Variation of infected human and infected mosquito populations
with different levels of larvicide: $c_A=0$,
$c_A=0.25$, $c_A=0.50$, $c_A=0.75$, and $c_A=1$.
The other controls are not applied: $c_{m} = 0$, and $\alpha = 1$.}
\end{center}
\end{figure}
\begin{figure}
\begin{center}
\subfloat[\footnotesize{Variation of infected individuals ($I_h$)
with mechanical control.}]{\label{16_human_mechanical_simulation}
\includegraphics[scale=0.45]{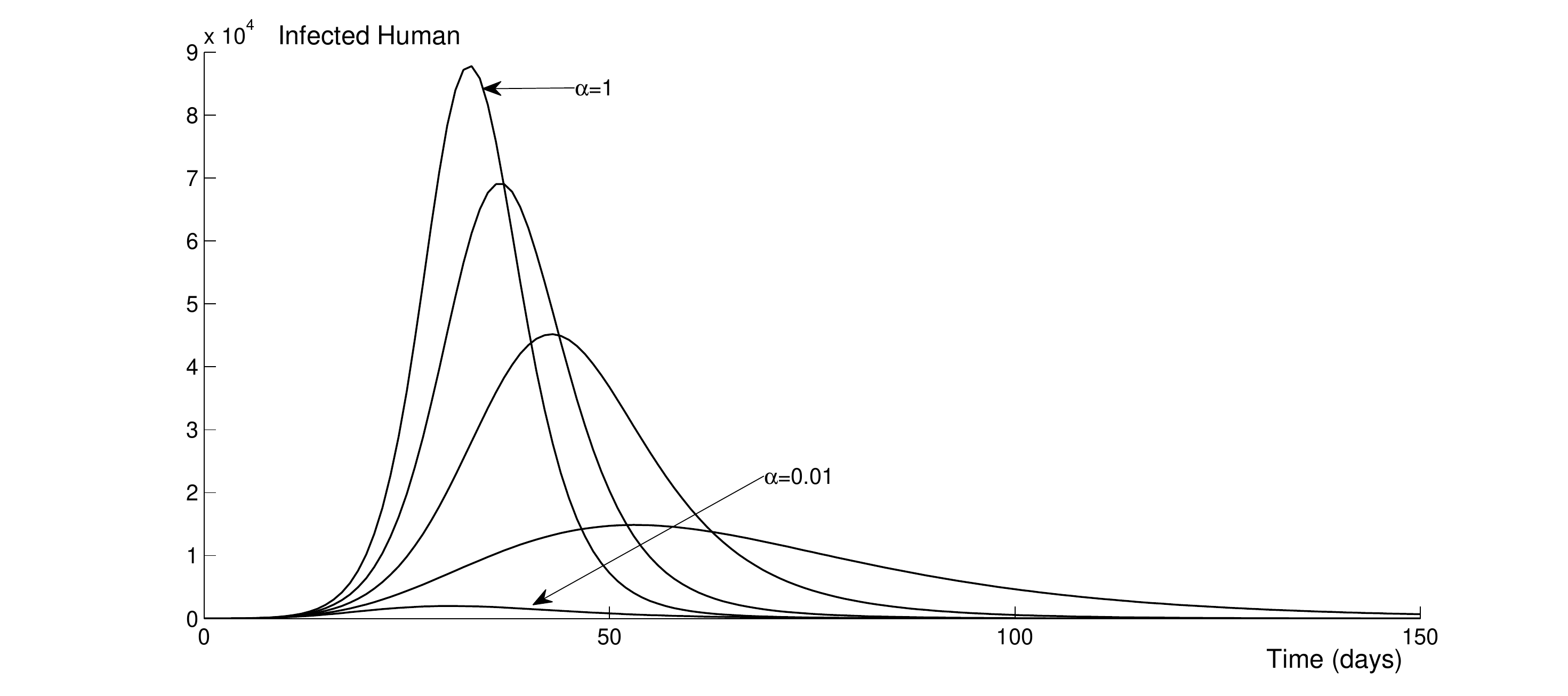}}
\newline
\subfloat[\footnotesize{Variation of infected mosquitoes ($I_m$) with
mechanical control.}]{\label{17_mosquito_mechanical_simulation}
\includegraphics[scale=0.45]{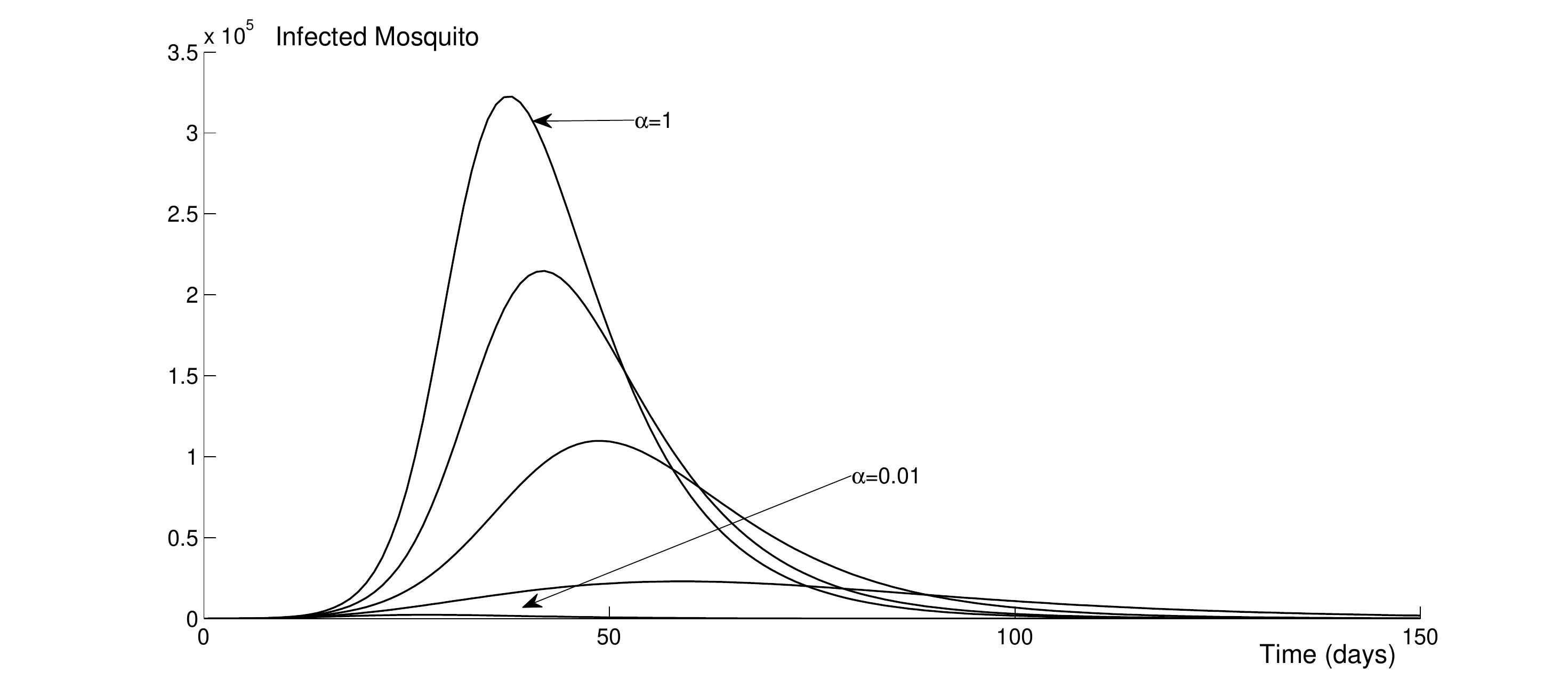}}
\caption{Variation of infected human and infected mosquito populations with
different levels of mechanical control: $\alpha = 0.01$, $\alpha = 0.25$,
$\alpha = 0.5$, $\alpha = 0.75$, and $\alpha = 1$.
The other controls are not applied: $c_{m} = 0$, and $c_A = 0$.}
\end{center}
\end{figure}

Figures~\ref{14_human_adulticide_simulation} and \ref{15_mosquito_adulticide_simulation}
on adulticide control, Figures~\ref{12_human_larvicide_simulation} and \ref{13_mosquito_larvicide_simulation}
on larvicide control, and Figures~\ref{16_human_mechanical_simulation} and
\ref{17_mosquito_mechanical_simulation} on mechanical control
show that a small quantity of each control is efficient to drop infection.
Figures~\ref{14_human_adulticide_simulation}
and \ref{15_mosquito_adulticide_simulation} show that
the human population is already well protected
by covering only 25\% of the country with insecticide
for adult mosquitoes. However, we consider that
\emph{Aedes aegypti} does not become resistant
to insecticide and that there is no shortage of insecticide.
Figures~\ref{12_human_larvicide_simulation},
\ref{13_mosquito_larvicide_simulation},
\ref{16_human_mechanical_simulation},
and \ref{17_mosquito_mechanical_simulation}
show the controls applied in the aquatic phase of the
mosquito. The controls were studied separately,
but each one is closely related to the other. None of these
controls is sufficient to drop the total number of infected
humans to zero, but the removal of breeding sites
and the use of larvicide contributes to prophylaxis.

Figures~\ref{22_human_mixed_control} and \ref{23_mosquitomixed}
show simulations using the three controls simultaneously.
They show that $10\%$ of each control, applied continuously,
is enough to contain the infection near zero case.
\begin{figure}
\begin{center}
\subfloat[\footnotesize{Variation of infected individuals ($I_h$)
with a combined use of the three controls.}]{\label{22_human_mixed_control}
\includegraphics[scale=0.45]{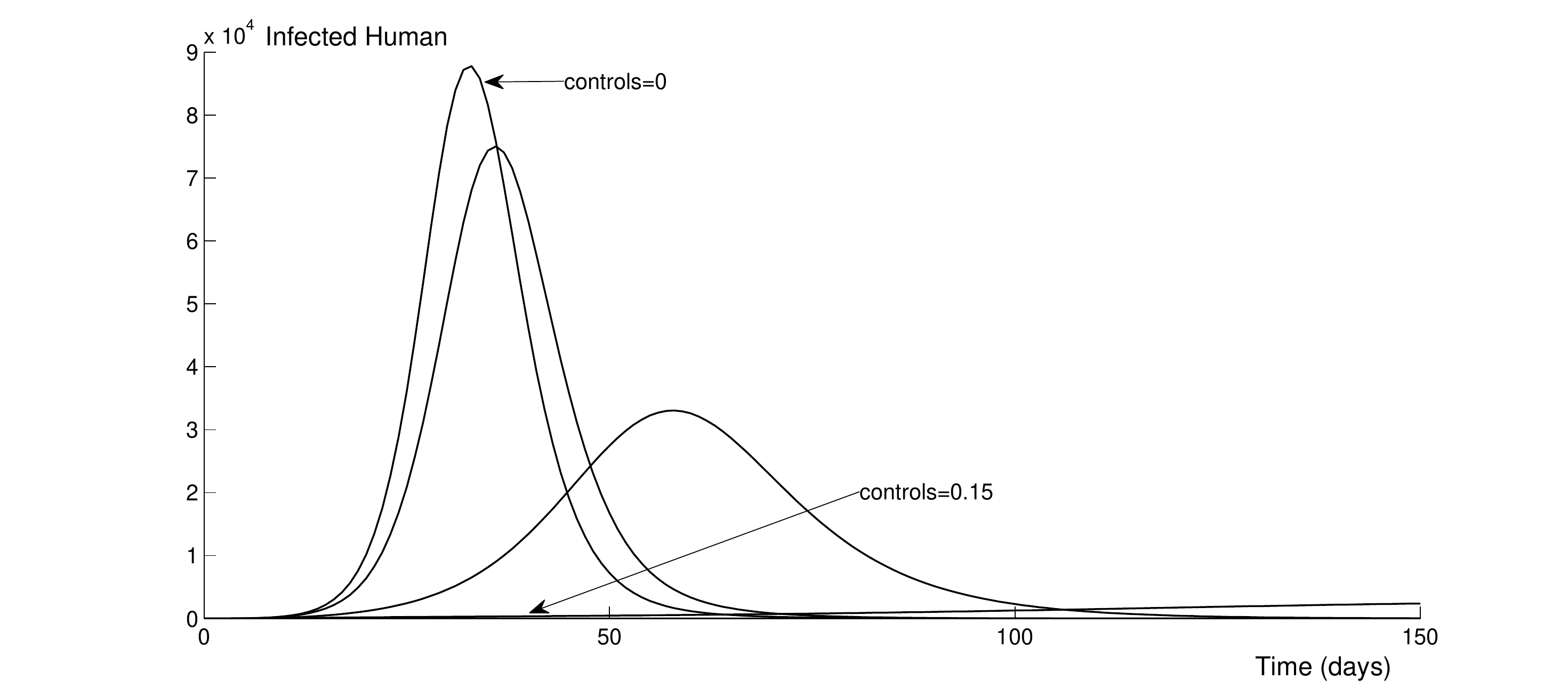}}
\newline
\subfloat[\footnotesize{Variation of infected mosquitoes ($I_m$)
with a combined use of the three controls.}]{\label{23_mosquitomixed}
\includegraphics[scale=0.45]{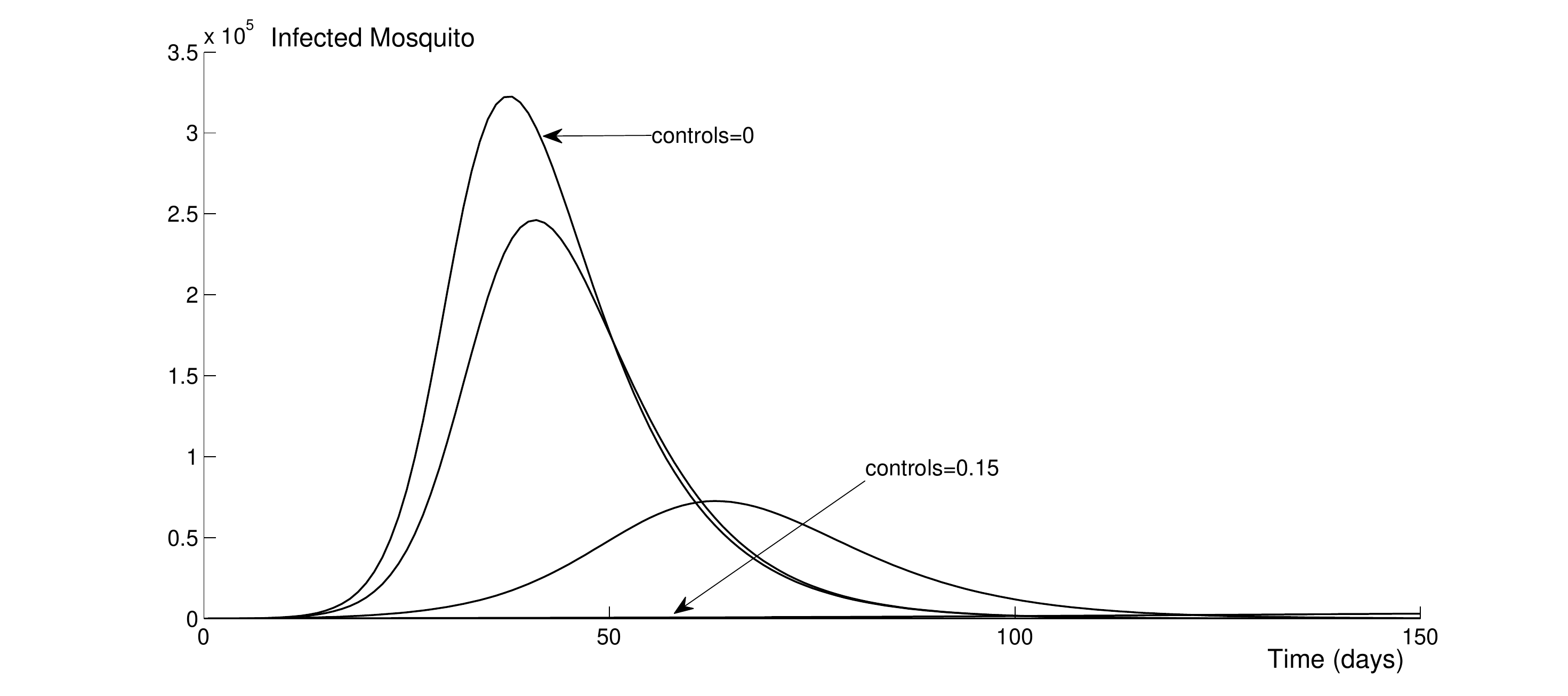}}
\caption{Variation of infected human and infected mosquito populations
when using larvicide, adulticide, and mechanical control simultaneously:
$c_A=c_m=1-\alpha= 0$, 0.01, 0.05, 0.1, 0.15.}
\end{center}
\end{figure}

The eradication of \emph{Aedes aegypti} may not be feasible
and, from the environment point of view, not desirable.
The aim is to reduce the mosquito density and increase
the immunity on the humans. Population herd immunity
can be reached by increasing the total number of resistant persons
to the disease, which implies that these persons have been infected,
or through vaccination.

No commercially available clinical cure or vaccine
is currently available for dengue, but efforts are to
develop one \citep{Blaney2007,WHO2007}. Effective vaccines
have been produced against other flavivirus diseases such as yellow fever,
Japanese encephalitis, and tick-borne encephalitis, so there is good hope
for a vaccine for dengue. We now simulate the epidemiological model with vaccination.


\section{Model with vaccination}
\label{sec:5}

While direct individual protection is the major focus of mass vaccination program,
population effects also contribute to individual protection
through herd immunity, providing protection for unprotected
individuals \citep{Farrington2003}. The more vaccinated people,
the less likely a susceptible person will come into contact
with the infection. With the introduction of a vaccine,
the SIR model related to the human population is augmented
into the SVIR model of Figure~\ref{modeloSVIR},
where $V_h$ represents the total number of vaccinated people.
\begin{figure}
\begin{center}
\includegraphics[scale=0.5]{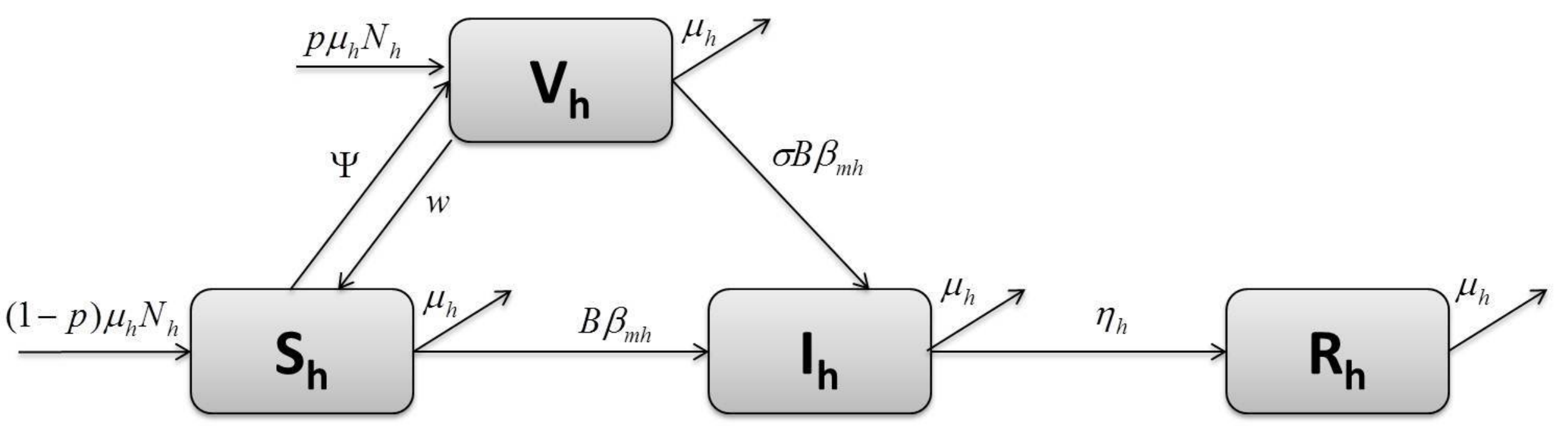}
\end{center}
\caption{\label{modeloSVIR} Epidemiological SVIR (Susceptible, Vaccinated, Infected, Recovered)
model for humans with $S_h$, $V_h$ (human vaccinated population),
$I_h$, and $R_h$ as state variables, $p$ (proportion of the vaccinated new born) as the control,
and $N_h$, $B$, $\beta_{mh}$, $\mu_h$, $\eta_h$,
$w$ (proportional rate at which vaccination loses effect),
$\Psi$ (fraction of the vaccinated susceptible),
and $\sigma$ (infection rate of vaccinated people) as given parameters.}
\end{figure}

Vaccination is continuous with a constant proportion $p$
of vaccinated new born. A fraction $\Psi$ of the susceptible
is vaccinated. The vaccination reduces but does not eliminate
susceptibility to infection. For this reason we consider
the infection rate $\sigma$ of vaccinated people:
when $\sigma=0$ the vaccine is perfect and when $\sigma=1$
the vaccine has no effect at all. The vaccination loses efficacy
at a rate $w$. The differential system for the host population is:
\begin{equation*}
\begin{cases}
S_h'(t) = (1-p)\mu_h N_h +w V_h(t) - \left(B\beta_{mh}\frac{I_m(t)}{N_h}+\Psi +\mu_h\right)S_h(t)\\
V_h'(t) = p\mu_h N_h+\Psi S_h(t)-\left(w+\sigma B \beta_{mh}\frac{I_m(t)}{N_h}+\mu_h\right)V_h(t)\\
I_h'(t) = B\beta_{mh}\frac{I_m(t)}{N_h}(S_h(t)+\sigma V_h(t)) -(\eta_h+\mu_h) I_h(t)\\
R_h'(t) = \eta_h I_h(t) - \mu_h R_h(t).
\end{cases}
\end{equation*}

Figure~\ref{18_human_vaccine_efficacy} shows
simulations for decreasing vaccine efficacy.
Larvicide, insecticide, and mechanical control
were kept null, and the parameters $\sigma$ and $w$ were changed.
For 80\% of the population vaccinated,
the efficacy of the vaccine reduces the disease spread.

\begin{figure}
\begin{center}
\includegraphics[scale=0.45]{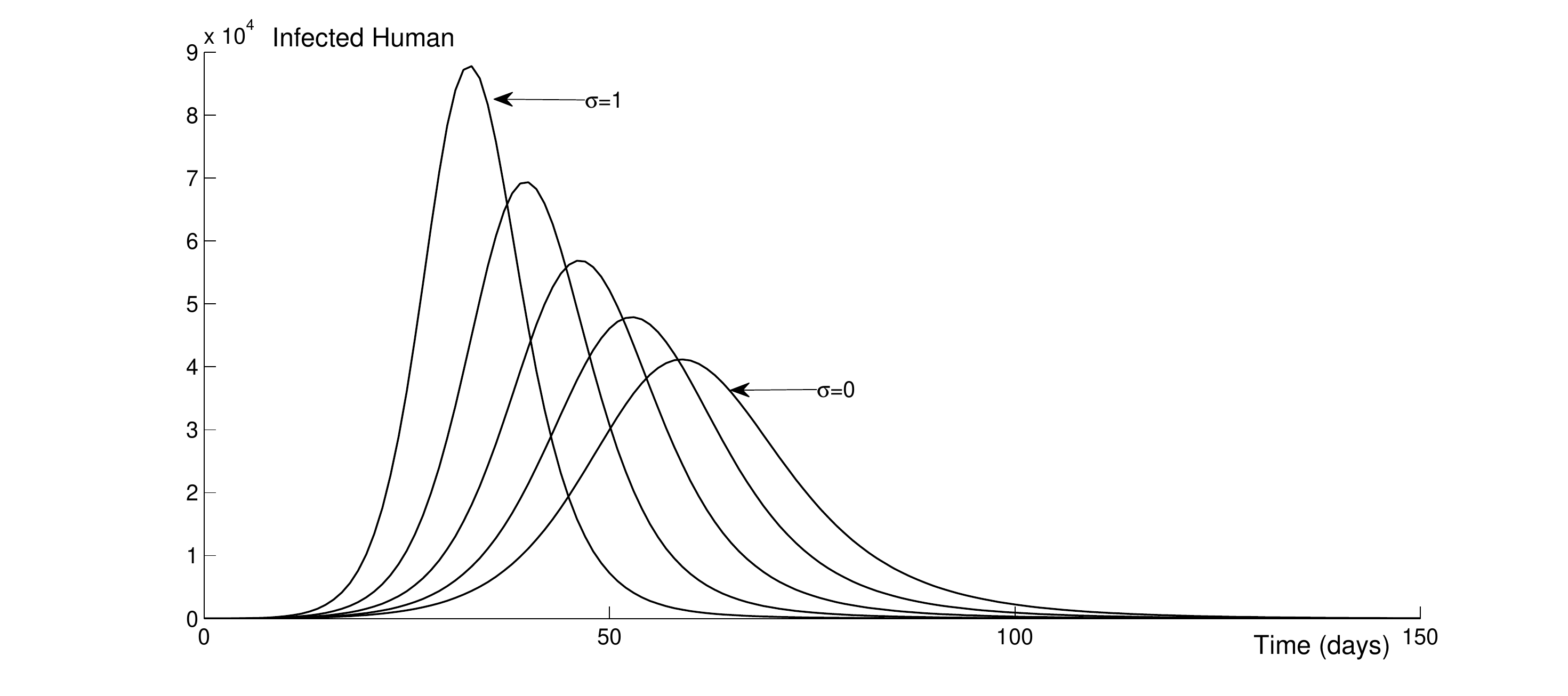}
\end{center}
\caption{\label{18_human_vaccine_efficacy} Total number of infected humans
for decreasing vaccine efficacy ($\sigma=1-w = 0$, 0.25, 0.5, 0.75, 1),
$p=0.80$, and $\Psi=0.80$.}
\end{figure}

\begin{figure}
\begin{center}
\includegraphics[scale=0.45]{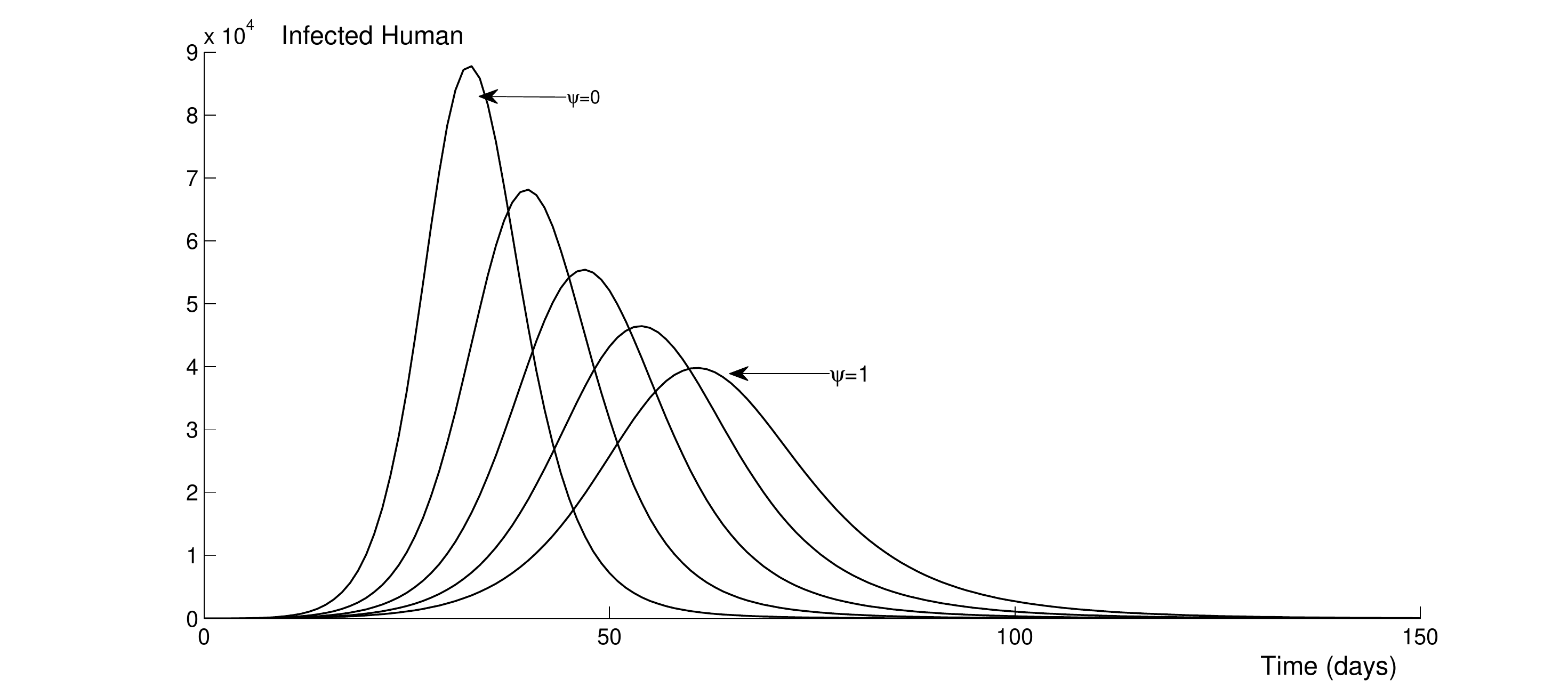}
\end{center}
\caption{\label{24_human_vaccine_susceptiblepopulation}
Total number of infected humans when using different values
of population vaccinated ($\Psi= 0$, 0.25, 0.5, 0.75, 1),
$p=0$, $w=0.85$, and $\sigma=0.15$.}
\end{figure}

Figure~\ref{24_human_vaccine_susceptiblepopulation} presents
the proportion of population vaccinated. It shows that
dengue fever prophylaxis articulates health and
sustainable development.


\section{Conclusion}
\label{sec:6}

Our simulations, based on our compartmental epidemiological model
and formalizing clean-up campaigns to remove the vector breeding sites
and the application of insecticides (larvicide and adulticide),
have shown that even with a low continuous control over time,
the results are encouraging. The model with an imperfect vaccine
has shown that the total number of infected persons can decrease quickly.


\section*{Acknowledgements}

This work was supported by FEDER funds through
COMPETE --- Operational Program Factors of Competitiveness
(``Programa Operacional Factores de Com\-pe\-ti\-ti\-vi\-da\-de'')
and by Portuguese funds through the
{\it Center for Research and Development
in Mathematics and Applications} (University of Aveiro),
the R\&D Center Algoritmi (University of Minho),
and the Portuguese Foundation for Science and Technology
(``FCT --- Funda\c{c}\~{a}o para a Ci\^{e}ncia e a Tecnologia''),
within project PEst-C/MAT/UI4106/2011
with COMPETE number FCOMP-01-0124-FEDER-022690.
Rodrigues was also supported by the Ph.D. fellowship SFRH/BD/33384/2008.



\end{document}